\documentclass[preprint,12pt]{elsarticle}

\usepackage{amsmath}
\usepackage{tikz}
\usepackage{paralist}
\usepackage{graphics}
\usepackage{amssymb}
\usepackage{graphicx}
\usepackage{epstopdf}
\newtheorem{thm}{Theorem}
\newtheorem{cor}[thm]{Corollary}
\newtheorem{lem}[thm]{Lemma}
\newtheorem{proof}[thm]{Proof}
\journal{https://arxiv.org}

\begin{document}

\begin{frontmatter}

%% Title, authors and addresses

%% use the tnoteref command within \title for footnotes;
%% use the tnotetext command for the associated footnote;
%% use the fnref command within \author or \address for footnotes;
%% use the fntext command for the associated footnote;
%% use the corref command within \author for corresponding author footnotes;
%% use the cortext command for the associated footnote;
%% use the ead command for the email address,
%% and the form \ead[url] for the home page:
%%
%% \title{Title\tnoteref{label1}}
%% \tnotetext[label1]{}
%% \author{Name\corref{cor1}\fnref{label2}}
%% \ead{email address}
%% \ead[url]{home page}
%% \fntext[label2]{}
%% \cortext[cor1]{}
%% \address{Address\fnref{label3}}
%% \fntext[label3]{}

\title{On the summablility of truncated double Fourier series}

%% use optional labels to link authors explicitly to addresses:
%% \author[label1,label2]{<author name>}
%% \address[label1]{<address>}
%% \address[label2]{<address>}

\author{Ahmed A. Abdelhakim}
\begin{abstract}
We estimate the
truncated double trigonometric  series \\$\;
\sum_{n=0}^{N}\sum_{m=0}^{M}a_{mn}
{e}^{2\pi \imath \left(m x+n y\right)},\,$
$ a_{mn} \in\mathbb{C}, $ in Lebesgue spaces
with mixed norms in terms of
the $p^{th}-q^{th}$ power finite double
sums of its coefficients. We obtain these estimates for
all possible values of the exponents
involved then we provide examples of
matrices in ${\mathbb{C}}^{M\times N}$
that maximize some of them up to a constant independent of
$M$ and $N$.
\end{abstract}
\end{frontmatter}

%%
%% Start line numbering here if you want
%%
% \linenumbers

%% main text
\section{The problem}
Let $1\leq p,\,q,\,r,\,s \leq \infty. $
Consider the Banach space
$l^{p,q}(M,N)$ of all complex matrices
$A\in \mathbb{C}^{M\times N}$ with the norm
\begin{equation*}
\parallel A\parallel_{l^{p,q}(M,N)}\,=:
\,\left\{
\begin{array}{ll}
\left(\sum_{n=1}^{N}
\left(\sum_{m=1}^{M} |a_{mn}|^{p}\right)^{q/p}
 \right)^{1/q}, & \hbox{$1\leq p,\,q < \infty$;} \\
\left(\sum_{n=1}^{N}  \left(\max_{ 1\leq m\leq M}|a_{mn}|\right)^{q}\right)^{1/q}, & \hbox{$p=\infty, $ $ 1\leq q < \infty$;} \\
 \max_{ 1\le n\leq N}
 \left(\sum_{m=1}^{M}
|a_{mn}|^{p}\right)^{1/p}, & \hbox{$1\leq p< \infty, $
$ q= \infty;$} \\
 \max_{\substack{\; 1\le m\leq M,\\1\le n\leq N
}}\, |a_{mn}|, & \hbox{$ p = \infty, $
$ q = \infty.$}
 \end{array} \right.
\end{equation*}
Consider in addition the Banach space
$\,L^{r,s}(\,[a,b];[c,d]\,)=:
 L^{q}\left(\,[c,d];L^{p}\left([a,b]
\,\right)\,\right)\,$ of all functions
$\,f: [a,b]\times [c,d]\rightarrow
\mathbb{C}\,$ that are Lebesgue measurable
on $\,[a,b]\times[c,d]\,$
and satisfy that $\,\parallel f \parallel_{L^{r,s}([a,b];[c,d])}\,=\,
\parallel\, \parallel f (x,y)\parallel_{L_{x}^{r}([a,b])} \,\parallel_{L_{y}^{s}([c,d])}\,<\,\infty.$\\\\
Let $\,T_{M,N}:l^{p,q}(M,N) \rightarrow L^{r,s}([0,1];[0,1])\,$ be the linear operator
that assigns to each matrix
$A\in l^{p,q}(M,N)$
the double trigonometric sum
 $S_{M,N} = T_{M,N} A $ defined  by
\begin{align}\label{trigsum}
S_{M,N}(x,y)=\sum_{n=1}^{N}\sum_{m=1}^{M}\,a_{mn}
\,{e}^{2\pi \imath \left((m-1)\,x+(n-1)\, y\right)}.
\end{align}
The function
$S_{M,N}$ is smooth and 1-periodic in each variable. If the complex entries $a_{mn}$
in (\ref{trigsum}) are the Fourier coefficients
of some function in $L^{1}([0,1]\times [0,1])$
then $S_{M,N}$
is a rectangular partial sum of the
double Fourier series of that function.
This truncated sum proved useful in
many applications (cf. \cite{baxan,samuel})
 Let $Q$  be the hypercube $[0,1]^{4}$
and, for simplicity, let
$\,l^{p,q},\,$
$\,L^{r,s}\,$ and  $\,\mathcal{C}_{M,N}(p,q,r,s)\, $
denote the spaces
$l^{p,q}(M,N)\, $ and
$\,L^{r,s}([0,1];[0,1])\,$
and the operator norm
$\,\displaystyle \parallel T_{M,N}\parallel_{l^{p,q}\rightarrow L^{r,s}}=\sup_{
\substack{A\in l^{p,q},\,A\neq O}}
{\left({\parallel T_{M,N}A \parallel_{L^{r,s}}}/
{\parallel A\parallel_{l^{p,q}}}\right)}\;$
respectively. We are interested in estimating $S_{M,N}$
in the mixed $L^{r,s}$ norm
in terms of the $l^{p,q}$ norm of
its coefficients matrix $A$.
That is, we would like to prove estimates of the form
\begin{align}\label{e}
\parallel S_{M,N}\parallel_{L^{r,s}}\,\leq\,
c_{M,N}(p,q,r,s)\,
\parallel A \parallel_{l^{p,q}}
\quad\text{for all points }\quad \left(\frac{1}{p},\frac{1}{q},\frac{1}{r},\frac{1}{s}
\right)\in Q.
\end{align}
Since the linear space $l^{p,q}$ is finite dimensional
then we guarantee not only the boundedness
of $\,T_{M,N}\,$ but also the existence of a
maximizing matrix $A_{p,q,r,s}\in
\mathbb{C}^{M\times N}$ for which
\begin{align*}
\parallel T_{M,N} A_{p,q,r,s}\parallel_{L^{r,s}}\;=\;
\mathcal{C}_{M,N}(p,q,r,s)\,\parallel A_{p,q,r,s} \parallel_{l^{p,q}.}
\end{align*}
Foschi \cite{damianorem} studied this kind of boundedness
for the one dimensional trigonometric sum
$\,\sum_{n=0}^{N-1}a_{n} e^{\imath n x}.$
In \cite{lova1}, Vukolova and Dyachenko
considered  the sums of double
trigonometric series in sines and cosines
with multiply monotonous coefficients
(see \cite{lova2} by the same authors) and
proved some estimates of these sums
in $L^{p}$  spaces with a mixed norm.
\section{$l^{p,q}-L^{r,s}$ estimates}
If we take absolute values of both sides of
(\ref{trigsum}) then apply the triangle inequality we easily get the estimate
\begin{align}\label{e1}
\parallel S_{M,N}\parallel_{L^{\infty,\infty}}
\;\leq\;\parallel A \parallel_{l^{1,1}.}
\end{align}
The set $O=\left\{e^{2\pi \imath\,(m\,x+n\,y)},
\,(m,n)\in \mathbb{Z}\times\mathbb{Z}\right\}$ is
an orthonormal system in $L^{2,2}.$ Thus
\begin{align*}
\parallel S_{M,N}\parallel^{2}_{L^{2,2}}
\,=&\,
\int_{0}^{1}\int_{0}^{1}|S_{M,N}(x,y)|^2 dx dy
=\;
\int_{0}^{1}\int_{0}^{1}S_{M,N}(x,y)
\overline{S_{M,N}(x,y)} dx dy\,=
\\=&\,
\sum_{k=1}^{N}
\sum_{j=1}^{M}
\sum_{n=1}^{N}
\sum_{m=1}^{M}
a_{jk} \overline{a_{m n}}
\int_{0}^{1}
e^{2\pi \imath\,(j-m)\,x} dx
\int_{0}^{1}
e^{2\pi \imath (k-n) y} dy\\
=&\,
\sum_{n=1}^{N}
\sum_{k=1}^{N}
\sum_{m=1}^{M}
\sum_{j=1}^{M}
a_{jk} \overline{a_{m n}}
\delta_{jm}\delta_{kn}
\,=\,\sum_{k=1}^{N}\sum_{j=1}^{M}
\,|a_{jk}|^{2}
\;=\;\parallel A \parallel^{2}_{l^{2,2}.}
\end{align*}
Hence we have
\begin{align}\label{e2}
\parallel S_{M,N}\parallel_{L^{2,2}}
\;=\;\parallel A \parallel_{l^{2,2}.}
\end{align}
To this end we can
obtain the estimates (\ref{e})
on $Q$ without further looking
at the properties of the operator $T_{M,N}.\,$
First observe that, by H\"{o}lder's inequality, we have
\begin{align}\label{holderf}
&\left.
\begin{array}{ll}
\parallel f \parallel_{L^{\bar{r},s}}\,\leq\,
\parallel f \parallel_{L^{r,s}},
& \qquad 1\leq \bar{r}  \leq r \leq \infty,\\\\
\parallel f \parallel_{L^{r,\bar{s}}}\,\leq\,
\parallel f \parallel_{L^{r,s}},
&\qquad 1\leq \bar{s}  \leq s \leq \infty,
\end{array}\right\}
\end{align}
for any function $f \in L^{r,s}.\;$
It also follows from H\"{o}lder's inequality that
\begin{align}\label{holdera}
&\left.\begin{array}{ll}
\parallel A \parallel_{l^{p,q}}\,\leq\,
M^{\frac{1}{p}-\frac{1}{\bar{p}}}\,
\parallel A \parallel_{l^{\bar{p},q}},
&\qquad 1\leq p \leq \bar{p} \leq \infty,\\\\
\parallel A \parallel_{l^{p,q}}\,\leq\,
N^{\frac{1}{q}-\frac{1}{\bar{q}}}\,
\parallel A \parallel_{l^{p,\bar{q}}},
& \qquad 1\leq q \leq \bar{q} \leq \infty,
\end{array}\right\}
\end{align}
for any $\, A\in \mathbb{C}^{M\times N}.\;$
Using (\ref{holderf})
we deduce from the estimate (\ref{e1}) that
$\,\parallel S_{M,N}\parallel_{L^{r,s}}$
$\;\leq\;\parallel A \parallel_{l^{1,1}},\quad
1\leq r,s\leq\infty,\,$
which tells us that all $\,L^{r,s}\,$ norms
of $\,S_{M,N}\,$ are controlled by the sum
$\;\sum_{n=1}^{N}\sum_{m=1}^{M}\,|a_{mn}|.\,$
Of course we could furthermore apply (\ref{holdera})
to the latter estimate and get that
$\;\parallel S_{M,N}\parallel_{L^{r,s}}
\;\leq\;M^{1-\frac{1}{p}}\,N^{1-\frac{1}{q}}\,\parallel A \parallel_{l^{p,q}},\quad
1\leq p,q,r,s\leq\infty.\,$ But we will momentarily find
stronger estimates  everywhere in $Q-\{(1,1,0,0)\}.\,$
For instance, if we apply (\ref{holderf}) to
the equality (\ref{e2}) we obtain the estimate
\begin{align}\label{e21}
\parallel S_{M,N}\parallel_{L^{r,s}}
\,\leq\,\parallel A \parallel_{l^{2,2}},
\quad \frac{1}{2}\leq \frac{1}{r} \leq 1,
\;\;\frac{1}{2}\leq \frac{1}{s}\leq 1,
\end{align}
and since, by (\ref{holdera}), $\displaystyle\:\parallel A \parallel_{l^{2,2}}\,\leq\, M^{\frac{1}{2}-\frac{1}{p}}
\,N^{\frac{1}{2}-\frac{1}{q}}\,
\parallel A \parallel_{l^{p,q}}\:$
for all $\displaystyle\:2\leq p,\,q \leq \infty\:$ then it follows from (\ref{e21}) that
\begin{align}\label{upper3}
\parallel S_{M,N}\parallel_{L^{r,s}}
\,\leq\,M^{\frac{1}{2}-\frac{1}{p}}
\,N^{\frac{1}{2}-\frac{1}{q}}\,
\parallel A \parallel_{l^{p,q}},
\quad 0\leq \frac{1}{p},\,
\frac{1}{q} \leq \frac{1}{2},\;\;
\frac{1}{2}\leq \frac{1}{r},\,\frac{1}{s} \leq 1.
\end{align}
By standard $L^{p}$ interpolation  (cf. \cite{BerghLofstrom}) between the $l^{ 1,1}-L^{\infty,\infty}$ estimate
(\ref{e1}) and the $l^{ 2,2}-L^{2,2}$ estimate
(\ref{e2}) we obtain
\begin{align}\label{e12}
\parallel S_{M,N}\parallel_{L^{r,s}}
\;\leq\;\parallel A \parallel_{l^{p,q}},
\end{align}
for all points $\, (\frac{1}{p},\frac{1}{q},
\frac{1}{r}, \frac{1}{s})\, $
on the line segment joining the two points
$(\frac{1}{2},\frac{1}{2},\frac{1}{2},\frac{1}{2})$ and
$(1,1,0,0)$  in $Q$ given by
$\;\;\;\displaystyle
\frac{1}{2}\leq \frac{1}{p}=\frac{1}{q}\leq 1,\quad
0\leq \frac{1}{r}=\frac{1}{s}\leq \frac{1}{2},\quad
\frac{1}{p}+\frac{1}{q}+\frac{1}{r}+\frac{1}{s}=2.$\\\\
Applying (\ref{holderf})  to the estimate (\ref{e12})
we get
\begin{align}
\label{u2} \hspace{-0.5 cm}\parallel S_{M,N}
\parallel_{L^{r,s}}
\,\leq\, \parallel A \parallel_{l^{p,q}},
\qquad
1-\frac{1}{p}\leq\frac{1}{r}\leq 1, \,\;
1-\frac{1}{q}\leq\frac{1}{s}\leq 1, \,\;
\frac{1}{2}\leq \frac{1}{p}=\frac{1}{q}\leq 1.
\end{align}
Moreover, using the first part of
(\ref{holdera}), it follows from (\ref{u2})
that \begin{align}
& \nonumber \hspace{-0.5 cm}\parallel S_{M,N}
\parallel_{L^{r,s}}
\,\leq\, M^{\frac{1}{q}-\frac{1}{p}}\,
\parallel A \parallel_{l^{p,q}},
\\ &\label{upper4} \qquad\qquad\qquad
1-\frac{1}{q}\leq\frac{1}{r}\leq 1, \,\;
1-\frac{1}{q}\leq\frac{1}{s}\leq 1, \,\;
0\leq \frac{1}{p}\leq \frac{1}{q}, \,\;
\frac{1}{2}\leq \frac{1}{q}\leq 1.
\end{align}
While applying the second part of
(\ref{holdera}) to (\ref{u2}) we obtain
\begin{align}
& \nonumber \hspace{-0.5 cm}\parallel S_{M,N}
\parallel_{L^{r,s}}
\,\leq\, N^{\frac{1}{p}-\frac{1}{q}}\,
\parallel A \parallel_{l^{p,q}},
\\ &\label{upper5} \qquad\qquad\qquad
1-\frac{1}{p}\leq\frac{1}{r}\leq 1, \,\;
1-\frac{1}{p}\leq\frac{1}{s}\leq 1, \,\;
0\leq \frac{1}{q}\leq \frac{1}{p}, \,\;
\frac{1}{2}\leq \frac{1}{p}\leq 1.
\end{align}
Observe here that (\ref{upper3}) follows
either from (\ref{upper4}) with $q=2$ after reapplying
(\ref{holdera}) to the norm $
\,\parallel A \parallel_{l^{p,2}}\,$ or
from (\ref{upper5}) with $p=2$ after reapplying
(\ref{holdera}) to the norm $
\,\parallel A \parallel_{l^{2,q}.}\;$\\
Interestingly, if we reverse the order in which we apply
the consequences of H\"{o}lder's inequality, (\ref{holderf})
and (\ref{holdera}), to the estimate (\ref{e12})
we recover the estimate
(\ref{e}) for another range of the exponents $\,p,q,r,s.\,$
Indeed, applying (\ref{holdera}) first to
the estimate (\ref{e12}) yields
\begin{align}
&\nonumber \hspace{-0.5 cm}\parallel S_{M,N}
\parallel_{L^{r,s}}
\,\leq\, M^{1-\frac{1}{r}-\frac{1}{p}}\,
N^{1-\frac{1}{s}-\frac{1}{q}}\,
\parallel A \parallel_{l^{p,q}},
\\ \label{u1}&\qquad\qquad\qquad\qquad \qquad
0\leq\frac{1}{p}\leq 1-\frac{1}{r}, \,\;
0\leq\frac{1}{q}\leq 1-\frac{1}{s}, \,\;
0\leq \frac{1}{r}=\frac{1}{s}\leq \frac{1}{2}.
\end{align}
Moreover, if we carefully use the inequalities
(\ref{holderf}) in (\ref{u1}) we are led to the estimates
\begin{align}
&\nonumber \parallel S_{M,N}
\parallel_{L^{r,s}}
\,\leq\, M^{1-\frac{1}{r}-\frac{1}{p}}\,
N^{1-\frac{1}{r}-\frac{1}{q}}\,
\parallel A \parallel_{l^{p,q}},
\\ &\label{upper21}\hspace{2 cm}
0\leq \frac{1}{p}\leq 1-\frac{1}{r}, \;\;
0\leq \frac{1}{q}\leq 1-\frac{1}{r}, \;\;
\frac{1}{r}\leq \frac{1}{s}\leq 1, \;\;
0\leq \frac{1}{r}\leq \frac{1}{2},\\
&\nonumber \parallel S_{M,N}
\parallel_{L^{r,s}}
\,\leq\, M^{1-\frac{1}{s}-\frac{1}{p}}\,
N^{1-\frac{1}{s}-\frac{1}{q}}\,
\parallel A \parallel_{l^{p,q}},
\\ &\label{upper22}\hspace{2 cm}
0\leq \frac{1}{p}\leq 1-\frac{1}{s}, \;\;
0\leq \frac{1}{q}\leq 1-\frac{1}{s}, \;\;
\frac{1}{s}\leq \frac{1}{r}\leq 1, \;\;
0\leq \frac{1}{s}\leq \frac{1}{2}.
\end{align}
Again, using  (\ref{holderf}),
the estimate (\ref{upper3})
results from (\ref{upper21})  with $\,r=2\,$
and from  (\ref{upper22}) with $\,s=2.\,$
The estimate (\ref{upper21}) coincides with
(\ref{upper4}) only in the region  $\displaystyle\,
{1}/{q}+{1}/{r}=1\,$  and
coincides with (\ref{upper5})
only in the region $\displaystyle\,
{1}/{p}+{1}/{r}=1.\,$ Similarly,
the estimate (\ref{upper22}) coincides with
(\ref{upper4}) and (\ref{upper5})
exclusively in the regions
$\displaystyle\,
{1}/{q}+{1}/{s}=1\,$  and
 $\displaystyle\,
{1}/{p}+{1}/{s}=1,\,$
respectively.
\\\\ The relation between the exponents $p,q,r,s$
for which the estimates
(\ref{upper4}), (\ref{upper5}),
(\ref{upper21}) and (\ref{upper22}) hold
can be demonstrated by the following respective
four sets of figures.
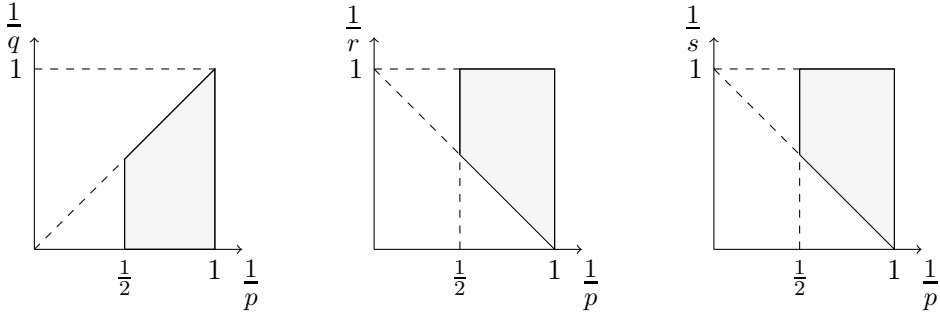
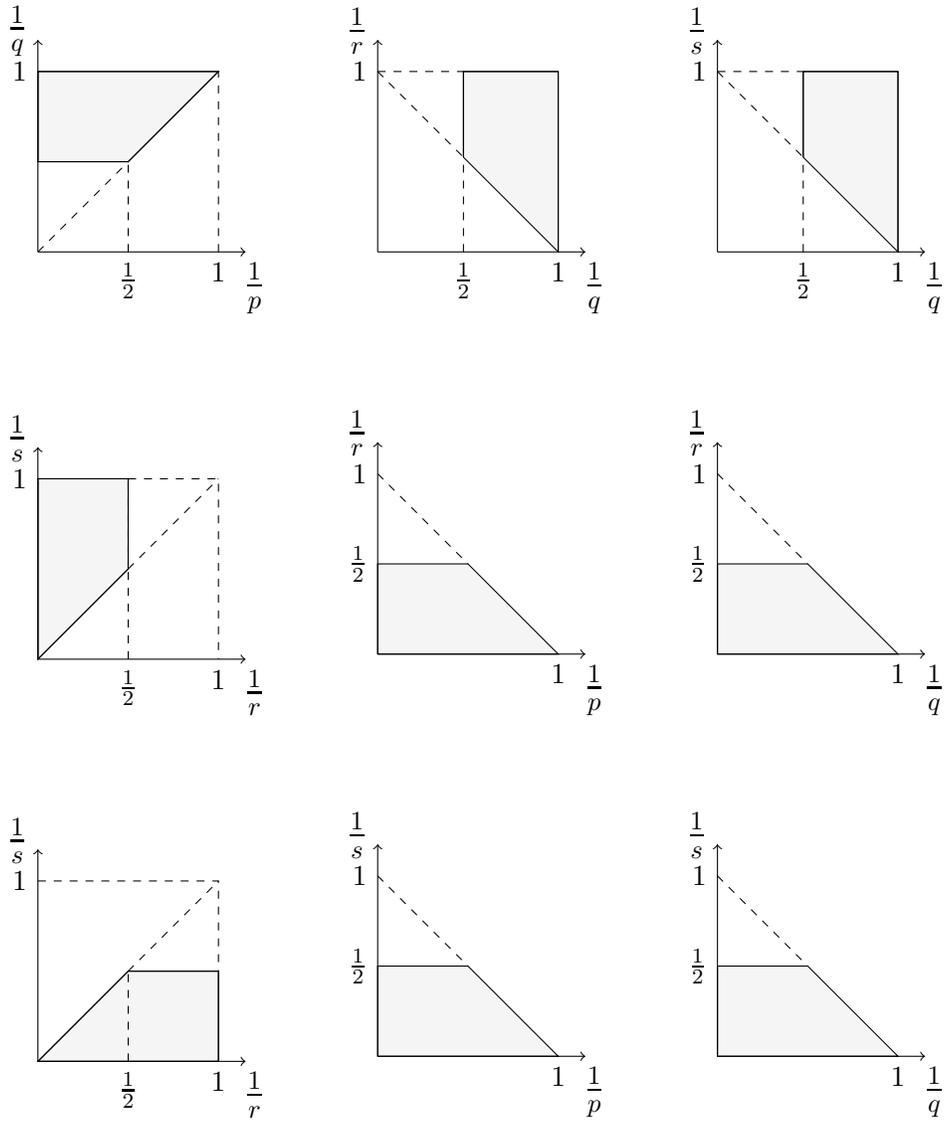
\begin{align*}
\hspace{-0.5 cm}
&\begin{tikzpicture} [scale=6]
\draw[fill=gray!8] (0.2,0.2)--(0.4, 0.4)--
(0.4,0.0)--(0.2,0.0)--cycle;
\draw (0.2,0.2)--(0.4, 0.4)--
(0.4,0.0)--(0.2,0.0)--cycle;
\draw[->] (0.0, 0.0) -- (0.46, 0);
\draw (0.48, 0) node[below] {\large{$\frac{1}{p}$}};
   \draw[->] (0.0,0.0) -- (0.0,0.47);
\draw (0.0,0.49) node[left] {\large{$\frac{1}{q}$}};
\draw (0.0,0.4) node[left] {\small{1}};
\draw (0.4,0.0) node[below] {\small{1}};
\draw (0.2, 0.0) node[below] {\small{$\frac{1}{2}$}};
\draw[dashed](0.0,0.4)--(0.4,0.4);
\draw[dashed](0.0,0.0)-- (0.2,0.2);
\end{tikzpicture}\qquad
\begin{tikzpicture} [scale=6]
\draw[fill=gray!8] (0.19,0.4)--(0.4,0.4)--
(0.4,0.0)--(0.19, 0.21)--cycle;
\draw[->] (0.0, 0.0) -- (0.46, 0);
\draw (0.48, 0) node[below] {\large{$\frac{1}{p}$}};
\draw[->] (0.0,0.0) -- (0.0,0.47);
\draw (0.0,0.49) node[left] {\large{$\frac{1}{r}$}};
\draw (0.0,0.4) node[left] {\small{1}};
\draw (0.4,0.0) node[below] {\small{1}};
\draw (0.19,0.21)--(0.19, 0.4);
\draw (0.19,0.0) node[below] {\small{$\frac{1}{2}$}};
\draw (0.4,0.0)--(0.4,0.4)--(0.19,0.4);
\draw [dashed](0.0,0.4)--(0.19,0.21);
\draw [dashed](0.0,0.4)--(0.19,0.4);
\draw [dashed] (0.19,0.0)--(0.19,0.21);
\end{tikzpicture}\qquad
\begin{tikzpicture} [scale=6]
\draw[fill=gray!8] (0.19,0.4)--(0.4,0.4)--
(0.4,0.0)--(0.19, 0.21)--cycle;
\draw[->] (0.0, 0.0) -- (0.46, 0);
\draw (0.48, 0) node[below] {\large{$\frac{1}{p}$}};
\draw[->] (0.0,0.0) -- (0.0,0.47);
\draw (0.0,0.49) node[left] {\large{$\frac{1}{s}$}};
\draw (0.0,0.4) node[left] {\small{1}};
\draw (0.4,0.0) node[below] {\small{1}};
\draw (0.19,0.21)--(0.19, 0.4);
\draw (0.19,0.0) node[below] {\small{$\frac{1}{2}$}};
\draw (0.4,0.0)--(0.4,0.4)--(0.19,0.4);
\draw [dashed](0.0,0.4)--(0.19,0.21);
\draw [dashed](0.0,0.4)--(0.19,0.4);
\draw [dashed] (0.19,0.0)--(0.19,0.21);
\end{tikzpicture}
\end{align*}
\begin{align*}
\begin{tikzpicture} [scale=6]
\draw[fill=gray!8] (0.2,0.2)--(0.4, 0.4)--
(0.0,0.4)--(0.0,0.2)--cycle;
\draw (0.2,0.2)--(0.4, 0.4)--
(0.0,0.4)--(0.0,0.2)--cycle;
\draw[->] (0.0, 0.0) -- (0.46, 0);
\draw (0.48, 0) node[below] {\large{$\frac{1}{p}$}};
   \draw[->] (0.0,0.0) -- (0.0,0.47);
\draw (0.0,0.49) node[left] {\large{$\frac{1}{q}$}};
\draw (0.0,0.4) node[left] {\small{1}};
\draw (0.4,0.0) node[below] {\small{1}};
\draw (0.2, 0.0) node[below] {\small{$\frac{1}{2}$}};
\draw[dashed](0.4,0.0)--(0.4,0.4);
\draw[dashed](0.0,0.0)-- (0.2,0.2);
\draw[dashed](0.2,0.0)-- (0.2,0.2);
\end{tikzpicture}\qquad
\begin{tikzpicture} [scale=6]
\draw[fill=gray!8] (0.19,0.4)--(0.4,0.4)--
(0.4,0.0)--(0.19, 0.21)--cycle;
\draw[->] (0.0, 0.0) -- (0.46, 0);
\draw (0.48, 0) node[below] {\large{$\frac{1}{q}$}};
\draw[->] (0.0,0.0) -- (0.0,0.47);
\draw (0.0,0.49) node[left] {\large{$\frac{1}{r}$}};
\draw (0.0,0.4) node[left] {\small{1}};
\draw (0.4,0.0) node[below] {\small{1}};
\draw (0.19,0.21)--(0.19, 0.4);
\draw (0.19,0.0) node[below] {\small{$\frac{1}{2}$}};
\draw (0.4,0.0)--(0.4,0.4)--(0.19,0.4);
\draw [dashed](0.0,0.4)--(0.19,0.21);
\draw [dashed](0.0,0.4)--(0.19,0.4);
\draw [dashed] (0.19,0.0)--(0.19,0.21);
\end{tikzpicture}\qquad
\begin{tikzpicture} [scale=6]
\draw[fill=gray!8] (0.19,0.4)--(0.4,0.4)--
(0.4,0.0)--(0.19, 0.21)--cycle;
\draw[->] (0.0, 0.0) -- (0.46, 0);
\draw (0.48, 0) node[below] {\large{$\frac{1}{q}$}};
\draw[->] (0.0,0.0) -- (0.0,0.47);
\draw (0.0,0.49) node[left] {\large{$\frac{1}{s}$}};
\draw (0.0,0.4) node[left] {\small{1}};
\draw (0.4,0.0) node[below] {\small{1}};
\draw (0.19,0.21)--(0.19, 0.4);
\draw (0.19,0.0) node[below] {\small{$\frac{1}{2}$}};
\draw (0.4,0.0)--(0.4,0.4)--(0.19,0.4);
\draw [dashed](0.0,0.4)--(0.19,0.21);
\draw [dashed](0.0,0.4)--(0.19,0.4);
\draw [dashed] (0.19,0.0)--(0.19,0.21);
\end{tikzpicture}
\end{align*}
\begin{align*}
\begin{tikzpicture} [scale=6]
\draw[fill=gray!8] (0.0,0.4)--(0.2, 0.4)--
(0.2,0.2)--(0.0,0.0)--cycle;
\draw (0.0,0.4)--(0.2, 0.4)--
(0.2,0.2)--(0.0,0.0)--cycle;
\draw[->] (0.0, 0.0) -- (0.46, 0);
\draw (0.48, 0) node[below] {\large{$\frac{1}{r}$}};
   \draw[->] (0.0,0.0) -- (0.0,0.47);
\draw (0.0,0.49) node[left] {\large{$\frac{1}{s}$}};
\draw (0.0,0.4) node[left] {\small{1}};
\draw (0.4,0.0) node[below] {\small{1}};
\draw (0.2, 0.0) node[below] {\small{$\frac{1}{2}$}};
\draw[dashed](0.2,0.4)--(0.4,0.4)--(0.4,0.0);
\draw[dashed] (0.2,0.2)--(0.4,0.4);
\draw[dashed] (0.2,0.2)--(0.2,0.4);
\draw[dashed] (0.2,0.0)--(0.2,0.2);
\end{tikzpicture}\qquad
\begin{tikzpicture} [scale=6]
\draw[fill=gray!8] (0.0,0.0)--(0.0,0.2)--
(0.2,0.2)--(0.4, 0.0)--cycle;
\draw[->] (0.0, 0.0) -- (0.46, 0);
\draw (0.48, 0) node[below] {\large{$\frac{1}{p}$}};
   \draw[->] (0.0,0.0) -- (0.0,0.47);
\draw (0.0,0.49) node[left] {\large{$\frac{1}{r}$}};
\draw (0.0,0.4) node[left] {\small{1}};
\draw (0.4,0.0) node[below] {\small{1}};
\draw (0.0,0.2) node[left] {\small{$\frac{1}{2}$}};
\draw [dashed] (0.0,0.4)--(0.2,0.2);
\end{tikzpicture}\qquad
\begin{tikzpicture} [scale=6]
\draw[fill=gray!8] (0.0,0.0)--(0.0,0.2)--
(0.2,0.2)--(0.4, 0.0)--cycle;
\draw[->] (0.0, 0.0) -- (0.46, 0);
\draw (0.48, 0) node[below] {\large{$\frac{1}{q}$}};
   \draw[->] (0.0,0.0) -- (0.0,0.47);
\draw (0.0,0.49) node[left] {\large{$\frac{1}{r}$}};
\draw (0.0,0.4) node[left] {\small{1}};
\draw (0.4,0.0) node[below] {\small{1}};
\draw (0.0,0.2) node[left] {\small{$\frac{1}{2}$}};
\draw [dashed] (0.0,0.4)--(0.2,0.2);
\end{tikzpicture}
\end{align*}
\begin{align*}
\begin{tikzpicture} [scale=6]
\draw[fill=gray!8] (0.4,0.0)--(0.4, 0.2)--
(0.2,0.2)--(0.0,0.0)--cycle;
\draw (0.4,0.0)--(0.4, 0.2)--
(0.2,0.2)--(0.0,0.0)--cycle;
\draw[->] (0.0, 0.0) -- (0.46, 0);
\draw[dashed](0.0,0.4)--(0.4,0.4)--(0.4,0.2);
\draw[dashed] (0.2,0.2)--(0.4,0.4);
\draw[dashed] (0.2,0.0)--(0.2,0.2);
\draw (0.48, 0) node[below] {\large{$\frac{1}{r}$}};
   \draw[->] (0.0,0.0) -- (0.0,0.47);
\draw (0.0,0.49) node[left] {\large{$\frac{1}{s}$}};
\draw (0.0,0.4) node[left] {\small{1}};
\draw (0.4,0.0) node[below] {\small{1}};
\draw (0.2, 0.0) node[below] {\small{$\frac{1}{2}$}};
\end{tikzpicture}\qquad
\begin{tikzpicture} [scale=6]
\draw[fill=gray!8] (0.0,0.0)--(0.0,0.2)--
(0.2,0.2)--(0.4, 0.0)--cycle;
\draw[->] (0.0, 0.0) -- (0.46, 0);
\draw (0.48, 0) node[below] {\large{$\frac{1}{p}$}};
   \draw[->] (0.0,0.0) -- (0.0,0.47);
\draw (0.0,0.49) node[left] {\large{$\frac{1}{s}$}};
\draw (0.0,0.4) node[left] {\small{1}};
\draw (0.4,0.0) node[below] {\small{1}};
\draw (0.0,0.2) node[left] {\small{$\frac{1}{2}$}};
\draw [dashed] (0.0,0.4)--(0.2,0.2);
\end{tikzpicture}\qquad
\begin{tikzpicture} [scale=6]
\draw[fill=gray!8] (0.0,0.0)--(0.0,0.2)--
(0.2,0.2)--(0.4, 0.0)--cycle;
\draw[->] (0.0, 0.0) -- (0.46, 0);
\draw (0.48, 0) node[below] {\large{$\frac{1}{q}$}};
   \draw[->] (0.0,0.0) -- (0.0,0.47);
\draw (0.0,0.49) node[left] {\large{$\frac{1}{s}$}};
\draw (0.0,0.4) node[left] {\small{1}};
\draw (0.4,0.0) node[below] {\small{1}};
\draw (0.0,0.2) node[left] {\small{$\frac{1}{2}$}};
\draw [dashed] (0.0,0.4)--(0.2,0.2);
\end{tikzpicture}
\end{align*}
One way to summarize the estimates obtained above
is the following theorem.
\begin{thm}\label{thmupper}
Let $\;\Theta: Q\rightarrow [\frac{1}{2},1]\;$
be the continuous surjection defined by\\
\begin{align*}
\Theta(\alpha,\beta,\gamma,\delta):=\,\left\{
  \begin{array}{ll}
\frac{1}{2}, & \hbox{$\;0\leq \alpha\leq \frac{1}{2},\;$
$\;0\leq \beta\leq \frac{1}{2},\;$ $\frac{1}{2}\leq
\gamma \geq 1,\;$
$\frac{1}{2}\leq
\delta \geq 1$;} \\\\
\alpha, & \hbox{$\;\frac{1}{2}\leq \alpha\leq 1,\;$
$\alpha\geq \beta,\;$ $\alpha+\gamma \geq 1,\;$
$\alpha+\delta \geq 1$;} \\\\
\beta, & \hbox{$\;\frac{1}{2}\leq \beta\leq 1,\;$
$\beta\geq \alpha,\;$ $\beta+\gamma \geq 1,\;$
$\beta+\delta \geq 1$;} \\\\
1-\gamma, & \hbox{$\;0\leq \gamma\leq \frac{1}{2},\;$
$\gamma\leq \delta,\;$ $\alpha+\gamma \leq 1,\;$
$\beta+\gamma \leq 1$;} \\\\
1-\delta, & \hbox{$\;0\leq \delta\leq \frac{1}{2},\;$
$\delta\leq \gamma,\;$ $\alpha+\delta \leq 1,\;$
$\beta+\delta\leq 1$.}
  \end{array}
\right.
\end{align*}
Then
\begin{align}\label{thsum}
\parallel S_{M,N}\parallel_{L^{r,s}}
\;\leq\;
\frac{\left(M N\right)^{\Theta\,
\left(\frac{1}{p},\frac{1}{q},\frac{1}{r},\frac{1}{s}\right
)}}{M^{\frac{1}{p}}
\,N^{\frac{1}{q}}}\,
\parallel A \parallel_{l^{p,q}.}
\end{align}
\end{thm}
Next, we try to find maximizers $ A_{p,q,r,s}\in
\mathbb{C}^{M\times N} $ for
the estimate (\ref{thsum}).
\section{Search for the maximizers}
We begin with discussing a potential maximizer
for the estimate (\ref{upper3}). We anticipate,
for this purpose, the estimate
(\ref{imm1}) in Lemma \ref{application} which is
an implication of Lemma \ref{van} below when $g$
is a constant function. Lemma \ref{van} is
due to Van der Corput. It provides an approximation
for exponential sums with certain phases by
oscillatory integrals.
\begin{lem}\label{van}
(\emph{\cite{marsh}, Lemma 4.10}). Let
$\,f\,$ be a smooth function such that $\,f^{\prime}(x)\,$ is decreasing with $\,f^{\prime}(a) = \beta,\,$ $\,f^{\prime}(b) = \alpha.\,$ Let $ g(x) $ be a real positive function with a continuous derivative
and $|g^{\prime}(x)|$ decreasing. Let $\,\theta\in(0,1)\,$ be a constant. Then
\begin{align*}
&\sum_{a<n< b}\,g(n)\,e^{2\pi \imath f(n)}
\;=\;\sum_{\alpha-\theta<\lambda< \beta+\theta}\,\int_{a}^{b}\,g(t)
\,e^{2\pi \imath \left(f(z)-\lambda z \right)}\,dz+\\
&\hspace{6 cm}
+O\left(g(a)\,\log{\left(\beta-\alpha+ 2\right)}\right)+
O\left(|g^{\prime}(a)|\right).
\end{align*}
\end{lem}
\begin{cor}\label{cor3}
If $f$ is a smooth real-valued function such that
$\;f^{\prime}\,$ is monotone and\hfill\\
$\,|f^{\prime}(x)|\leq \theta <1\,$ then the exponential sum
\begin{align}\label{oscsum}
&\sum_{a<n< b}\,e^{2\pi \imath f(n)}
\;=\;\int_{a}^{b}\,e^{2\pi \imath f(z)}\,dz+
O\left(1\right).
\end{align}
\end{cor}
Exploiting the assertion of Corollary
\ref{cor3} we get to prove the following lemma.
\begin{lem}\label{application}
Let $\;M >> 1\;$ and let $\;x\in [\eta,1-\eta]\;$ for some fixed $\;0<\eta <1.$ Then
\begin{align}\label{imm}
\sum_{m=0}^{M-1}\,e^{2\pi \imath m\,
\left( x-\frac{\eta}{4}\,\frac{m}{M}\right)}
\;=\;\sqrt{\frac{2}{\eta}}\,e^{-\frac{\pi}{4} \imath }\,e^{\frac{2\pi \imath M}{\eta}\,x^2}\,
\sqrt{M}+O(1)
\end{align}
so that
\begin{align}\label{imm1}
\left|\sum_{m=0}^{M-1}\,e^{2\pi \imath m\,
\left( x-\frac{\eta}{4}\,\frac{m}{M}\right)}\right|
\;\gtrsim\;\sqrt{M}.
\end{align}
\end{lem}
\begin{proof}
Since $\,z\longmapsto \,z\,
\left( x-\frac{\eta}{4}\,\frac{z}{M}\right)\,$
has a monotonically decreasing first derivative
bounded by $\,1-\frac{\eta}{2}\,$
when $0\leq z \leq M$ and $0\leq x\leq 1-\eta,\,$ then,
by (\ref{oscsum}) of Corollary \ref{cor3}, we have
\begin{align*}
\sum_{m=0}^{M-1}\,e^{2\pi \imath m\,
\left( x-\frac{\eta}{4}\,\frac{m}{M}\right)}
\,=\,I({M})+
O\left(1\right),\quad I({M})=\int_{0}^{M-1}
\,e^{2\pi \imath z\,
\left( x-\frac{\eta}{4}\,\frac{z}{M}\right)}\,dz.
\end{align*}
Next, we manipulate the oscillatory integral $I({M})$
as follows
\begin{description}
  \item[(i)] Rescale $\,z\rightarrow M\,z$
then  complete the square for the quadratic polynomial in $z$ in the exponent
  \item[(ii)] Rewrite the resulting integral
$\;\displaystyle \int_{0}^{1-{1}/{M}}=
\int_{-\infty}^{+\infty}-\int_{-\infty}^{0}
-\int_{1-{1}/{M}}^{+\infty}.$
\item[(iii)] For the integral
   $\;\displaystyle \int_{-\infty}^{+\infty},\,$
translate  $\,z\rightarrow z+\frac{2}{\eta}\,x\,$
then apply the formula (cf. \cite{taobook}, Section 2.2) $\;\displaystyle
   \int_{-\infty}^{+\infty}\,e^{-\imath a z^{2}}\,d z =\sqrt{\frac{\pi}{a}}\,e^{-\frac{\pi}{4}\imath},
\;$ $a>0,\;$ that can be justified by a contour integral.
\item[(iv)] Show that both integrals
$\;\displaystyle \int_{-\infty}^{0},\,$
$\;\displaystyle \int_{1-{1}/{M}}^{+\infty}$
are at most $\,O(\eta^{-1}M^{-1})\,$ via integration by parts.
\end{description}
Following the steps \textbf{(i)} - \textbf{(iv)} we see that
\begin{align}\label{im0}
\hspace{-0.5 cm}
I({M}) = M\int_{0}^{1-\frac{1}{M}}
\,e^{2\pi \imath M\,z\,\left( x-\frac{\eta}{4}\,z\right)}\,dz
= M e^{\frac{2\pi \imath M}{\eta}\,x^2}\,
\int_{0}^{1-\frac{1}{M}}
\,e^{-\frac{\pi}{2} \imath \,\eta M\,\left( z-\frac{2}{\eta}\,x\right)^{2}}\,dz.
\end{align}
Now, we have
\begin{align}\label{im1}
\hspace{-0.5 cm}\int_{-\infty}^{+\infty}
\,e^{-\frac{\pi}{2} \imath \,\eta M\,\left( z-\frac{2}{\eta}\,x\right)^{2}}\,dz=
\int_{-\infty}^{+\infty}
\,e^{-\frac{\pi}{2} \imath \,\eta M\,z^{2}}\,dz = \frac{\sqrt{2}\,e^{-\frac{\pi}{4} \imath }}{\sqrt{\eta\,M}}.
\end{align}
Also
\begin{align}
\nonumber &- \pi \imath \eta M\int_{-\infty}^{0}
\,e^{-\frac{\pi}{2} \imath \,\eta M\,\left( z-\frac{2}{\eta}\,x\right)^{2}}\,dz=\,
\int_{-\infty}^{0}
\frac{\partial_{z}e^{-\frac{\pi}{2} \imath \,\eta M\,\left( z-\frac{2}{\eta}\,x\right)^{2}}}{z-\frac{2}{\eta}\,x}
\,d z\\
\label{im2}&\qquad\qquad=\,\frac{e^{-\frac{\pi}{2} \imath \,\eta M\,\left( z-\frac{2}{\eta}\,x\right)^{2}}}{z-\frac{2}{\eta}\,x}
\bigg|_{-\infty}^{0}+
\int_{-\infty}^{0}
\frac{e^{-\frac{\pi}{2}
 \imath \,\eta M\,\left( z-\frac{2}{\eta}\,x\right)^{2}}}{
 \left(z-\frac{2}{\eta}\,x\right)^{2}}\,dz = O(1),
\end{align}
because when $\;\eta \leq x\leq 1-\eta$
\begin{align*}
\left|\int_{-\infty}^{0}
\frac{e^{-\frac{\pi}{2} \imath \,\eta M\,\left( z-\frac{2}{\eta}\,x\right)^{2}}}{\left(z-\frac{2}{\eta}\,x\right)^{2}}\,dz\right|
\leq
\int_{-\infty}^{0}
{\left(z-\frac{2}{\eta}\,x\right)^{-2}}\,dz
=
\frac{-1}{z-\frac{2}{\eta}\,x}\bigg|_{-\infty}^{0}
\in \left[{\eta}/{\left(2(1-\eta)\right)},{1}/{2}\right].
\end{align*}
Similarly, one can  verify that
for $\;\eta \leq x \leq 1-\eta,$
\begin{align}\label{im3}
\eta M\int_{1-{1}/{M}}^{+\infty}
\,e^{-\frac{\pi}{2} \imath \,\eta M\,\left( z-\frac{2}{\eta}\,x\right)^{2}}\,dz=O(1).
\end{align}
The identity (\ref{im0}) together with the estimates (\ref{im1}) - (\ref{im3}) yield (\ref{imm}).
\end{proof}
Now, let $\,B \in \mathbb{C}^{M\times N}\,$ be such that
$\;b_{jk}=\,e^{-\frac{\pi}{2}\,\eta\,\imath\,
\left({(j-1)^{2}}/{M}+ {(k-1)^{2}}/{N}\right)}\;$
so that
\begin{align}
\nonumber T_{M,N}B\,(x,y)=&\,
\sum_{k=1}^{N}\sum_{j=1}^{M}\,
e^{-\frac{\pi}{2}\,\eta\,\imath\,
\left({(j-1)^{2}}/{M}+ {(k-1)^{2}}/{N}\right)}
\,{e}^{2\pi \imath \left((j-1)\,x+(k-1)\, y\right)}\\
\label{tdxy}=&\,
\left(\sum_{j=0}^{M-1}\,e^{2\pi \imath j\,
\left( x-\frac{\eta}{4}\,\frac{j}{M}\right)}\right)
\left(\sum_{k=0}^{N-1}\,e^{2\pi \imath k\,
\left( x-\frac{\eta}{4}\,\frac{k}{N}\right)}\right).
\end{align}
For $\,(x,y)\in [\eta,1-\eta]^{2},\,$ we can
use (\ref{imm1}) to estimate the
exponential sums in (\ref{tdxy}) and obtain
$\;\displaystyle |T_{M,N}B\,(x,y)|\,\gtrsim\,M^{\frac{1}{2}}\,
N^{\frac{1}{2}}.\,$ This implies that
$\;\parallel T_{M,N}\, B \parallel_{L^{r,s}}
\,\gtrsim\,M^{\frac{1}{2}}\,
N^{\frac{1}{2}}.\,$ And evidently $\;\parallel  B \parallel_{l^{p,q}}=
M^{\frac{1}{p}}\,N^{\frac{1}{q}}.$ Thus
\begin{align}\label{l3}
\parallel T_{M,N}\, B \parallel_{L^{r,s}}\,\gtrsim\,M^{\frac{1}{2}-\frac{1}{p}}\,
N^{\frac{1}{2}-\frac{1}{q}}\, \parallel  B \parallel_{l^{p,q}.}
\end{align}
Estimate (\ref{l3}) stands behind
our intuition that the matrix $B$ is
a candidate maximizer for (\ref{upper3}).
Let $ C\in {\mathbb{C}}^{M\times N} $ be a column matrix
of ones so that $c_{jk}=1$ for some fixed
 $1\leq k \leq N,\,$
$c_{j\bar{k}}=0,\,$  $ \bar{k}\neq k.$ Then
\begin{align}
\label{eqc1} T_{M,N}\,C\, (x,y)\;&=\;
e^{2\pi \imath (k-1)\, y}\,
\sum_{m=1}^{M}\,e^{2\pi \imath (m-1)\, x}
\;=\;
e^{2\pi \imath (k-1)\, y}\,
e^{\pi \imath
\left(M-1\right)\,x}
\,\frac{\sin{\left(\pi M x\right)}}{
\sin{\left(\pi x\right)}}.
\end{align}
Since $\;\displaystyle \sin{(1)}\leq \frac{\sin{(\theta)}}{\theta}\leq 1\;$
whenever $\;0\leq \theta \leq 1.\;$ Then
\begin{align*}
\frac{\sin{\left(\pi M x\right)}}{\pi M x}\,\geq\,
\sin{(1)},\quad
\frac{\sin{\left(\pi x\right)}}{\pi  x}\,\leq\,{1},
\quad\text{whenever}\quad  0\leq x\leq \frac{1}{\pi M}.
\end{align*}
Hence
$\;\displaystyle \frac{\sin{\left(\pi M x\right)}}{\sin{
\left(\pi x\right)}}\,\geq\,\sin{(1)}\,M\;$
when $\;\displaystyle 0\leq x\leq \frac{1}{\pi M}.$
Applying this inequality to
(\ref{eqc1}), taking into account that
$\;e^{\pi \imath \,\left(
\left(M-1\right)\,x+2\left(k-1\right)\,y\right)}\;$ 
is a unit vector in $\mathbb{C},\,$
we get
\begin{align*}
|T_{M,N}\,C\, (x,y)|\;\geq\;
\sin{(1)}\,M,\qquad x \in[0,{1}/{\pi M}].
\end{align*}
Turning to the $\,L^{r,s}\,$ norm,
the latter estimate implies
\begin{align*}
\parallel T_{M,N}\,C \parallel_{L^{r,s}}\;\geq\;
\parallel T_{M,N}\,C \parallel_{L^{r,s}
\left(\,[0,{1}/{\pi M}];\,[0,1]\,\right)}\;\geq\;
\frac{\sin{(1)}}{\pi^{\frac{1}{r}}}
\,M^{1-\frac{1}{r}}.
\end{align*}
But  $\,\parallel C\parallel_{l^{p,q}} \,= \,
\,M^{\frac{1}{p}}.\,$ Therefore
\begin{align}\label{cl2}
\parallel T_{M,N}\,C \parallel_{L^{r,s}}\;\geq\;
\frac{\sin{(1)}}{\pi^{\frac{1}{r}}}
\,M^{1-\frac{1}{r}-\frac{1}{p}}\,
\parallel C\parallel_{l^{p,q}.}
\end{align}
We deduce from (\ref{cl2}) that, up to the constant 
$\,{\sin{(1)}}/{\pi^{\frac{1}{r}}},\,$  the
vector $C$ maximizes both estimates (\ref{upper4})
and (\ref{upper21}) in their region of coincidence
$\displaystyle\,
{1}/{q}+{1}/{r}=1.$
Analogously, if $R\in {\mathbb{C}}^{M\times N} $
is a row matrix
of ones then $\,\parallel R\parallel_{l^{p,q}}=
N^{\frac{1}{q}}\,$ and we have
\begin{align}\label{rl2}
\parallel T_{M,N}\,R \parallel_{L^{r,s}}\;\geq\;
\parallel T_{M,N}\,R \parallel_{L^{r,s}
\left(\,[0,{1}];\,[0,{1}/{\pi N}]\,\right)}
\;\geq\;\frac{\sin{(1)}}{\pi^{\frac{1}{s}}}
\,N^{1-\frac{1}{s}-\frac{1}{q}}\,
\parallel R\parallel_{l^{p,q}.}
\end{align}
From (\ref{rl2}) we see that, up to the constant
$\,{\sin{(1)}}/{\pi^{\frac{1}{s}}},\,$  the
row matrix $R$ maximizes both estimates (\ref{upper5})
and (\ref{upper22}) in the region $\displaystyle\,
{1}/{p}+{1}/{s}=1.$\\
Furthermore, if $D\in {\mathbb{C}}^{M\times N}$
is a matrix of ones then it is easy to verify
in the same spirit that
\begin{align}\label{l2}
\parallel T_{M,N}\,D \parallel_{L^{r,s}}\;\geq\;
\frac{\sin^{2}{(1)}}{\pi^{\frac{1}{r}+\frac{1}{s}}}
\,M^{1-\frac{1}{r}-\frac{1}{p}}
\,N^{1-\frac{1}{s}-\frac{1}{q}}\,\parallel D\parallel_{l^{p,q}.}
\end{align}
The inequality (\ref{l2})
shows that the matrix $D$ maximizes
the estimate (\ref{u1}) up to the constant
$\,{\sin^{2}{(1)}}/{\pi^{\frac{1}{r}+\frac{1}{s}}}$.\\
Notice that we can achieve (\ref{cl2}) applying
the same argument if the nonzero entries, the ones, 
in  $C$ are replaced by an
arbitrary complex constant. The same claim
holds for the matrices  $R$ and $D$. \\
Finally, let $\;E\in {\mathbb{C}}^{M\times N}$
be such that
$\,e_{mn}=1\,$ for some
$1\leq m\leq M,\,$ $1\leq n\leq N\,$
and $\,e_{jk}=0\,$  for
$ j\neq m,\,$ $k\neq n.$
Obviously $\,\parallel E\parallel_{l^{p,q}}=1\,$ and
since $\,|T_{M,N}\,E\, (x,y) |=1\,$ then
we have $\,\parallel T_{M,N}\, E\parallel_{L^{r,s}}=1\,$
as well for all values of Lebesgue exponents $\,p,q,r,s.$
So
\begin{align}\label{l1}
\parallel T_{M,N}\, E\parallel_{L^{r,s}}\,=\,
\parallel E\parallel_{l^{p,q}.}
\end{align}
Observe here that not only does the nonzero entry $e_{mn}$
enjoy an arbitrary position but it can also be taken
to be an arbitrary complex constant.
We would always have the equality (\ref{l1}).
From (\ref{l1}) we realize
that the estimate (\ref{u2}) is sharp
and that the matrix $E$
is a maximizer for it.
\section{Asymptotic behaviour of
$\;\parallel T_{M,N}\parallel_{l^{p,q}\rightarrow L^{r,s}}\;$ as $\,M,N\rightarrow \infty.$}
On one hand the inequalities
(\ref{upper3}), (\ref{upper4}), (\ref{upper5}),
(\ref{upper21}) and (\ref{upper22}) provide upper bounds
for the positive constants $\,\mathcal{C}_{M,N}(p,q,r,s)\,$
in the hypercube $Q$. While on the other
 hand each of the estimates
(\ref{l3}), (\ref{cl2}) - (\ref{l1})  gives
a lower bound for them in a
certain range of exponents values
as pointed out in Section 3.
Putting these bounds together,
we can describe the asymptotic
behaviour of $\,\mathcal{C}_{M,N}(p,q,r,s)\,$ as
$\,M,N\rightarrow \infty\,$ in some regions in $Q$.
\begin{thm}
Let $\,\Phi:Q\rightarrow[\frac{1}{2},1]\,$ be the restriction of the continuous surjection
$\Theta$ in Theorem \ref{thmupper} defined by
\begin{align*}
\hspace{-0.7 cm}
\Phi(\alpha,\beta,\gamma,\delta):=
\,\left\{
  \begin{array}{ll}
\frac{1}{2}, & \hbox{
$0\leq \alpha\leq \frac{1}{2},$
$0\leq \beta\leq \frac{1}{2},$
$\frac{1}{2}\leq \gamma \geq 1,$
$\frac{1}{2}\leq
\delta \geq 1$;} \\\\
\alpha, & \hbox{$\;\frac{1}{2}\leq \alpha\leq 1,\;$
$\alpha\geq \beta,\;$ $\alpha+\gamma \geq 1,\;$
$\alpha+\delta = 1$;} \\\\
\sqrt{\alpha\beta}, &
 \hbox{$\;\frac{1}{2}\leq \alpha\leq 1,\;$
$\alpha=\beta,\;$ $\alpha+\gamma > 1,\;$
$\alpha+\delta > 1$;} \\\\
\beta, & \hbox{$\;\frac{1}{2}\leq \beta\leq 1,\;$
$\beta\geq \alpha,\;$ $\beta+\gamma = 1,\;$
$\beta+\delta \geq 1$;} \\\\
\sqrt{\left(1-\gamma\right)
\left(1-\delta\right)}
, & \hbox{$\;0\leq \gamma\leq \frac{1}{2},\;$
$\gamma=\delta,\;$ $\alpha+\gamma \leq 1,\;$
$\beta+\delta\leq 1$.}
  \end{array}
\right.
\end{align*}
Then
\begin{align*}
\mathcal{C}_{M,N}(p,q,r,s)\,\sim\,
\frac{\left(M N\right)^{\Phi\,
\left(\frac{1}{p},\frac{1}{q},\frac{1}{r},\frac{1}{s}\right
)}}{M^{\frac{1}{p}}
\,N^{\frac{1}{q}}}.
\end{align*}
\end{thm}
\section{The estimates when the sum terms are not
orthogonal in $L^{2,2}$}
We have seen in Section 2 how we obtained
all $l^{p,q}-L^{r,s}$ estimates from
the $l^{1,1}-L^{\infty,\infty}$
and $l^{ 2,2}-L^{2,2}$ estimates with the help of the
powerful tool of interpolation and  manipulation of
H\"{o}lder's inequality. In order to get the equality
(\ref{e2}) we relied on the orthogonality
of the terms of the double trigonometric
sum in (\ref{trigsum}) in $L^{2,2}([0,1];[0,1]).$\\
Nevertheless, we can still
prove estimates of the form
$\,\parallel S_{M,N}\parallel_{L^{2,2}}
\;\lesssim\;\parallel A \parallel_{l^{2,2}}\,$
when the oscillatory terms in $S_{M,N}(x,y)$
are not orthogonal.\\
To illustrate the idea, assume that
 $\;V_{M,N}(x,y)=\sum_{n=1}^{N}\sum_{m=1}^{M}\,a_{mn}
\,{e}^{\imath \left((m-1)\,x+(n-1)\, y\right)}.$
By Urysohn's lemma \cite{analysis},
let $\chi$ be a nonnegative cutoff function
supported in $]-1/4,5/4[\,$ such that
$\chi(x)=1$ on $\,[0,1].$ Then
\begin{align}
&\nonumber
\parallel V_{M,N}\parallel^{2}_{L^{2,2}}\,\leq\,
\int_{\mathbb{R}}\int_{\mathbb{R}}\,\chi(x)\,\chi(y)\,
| V_{M,N}(x,y)|^{2}\,dx dy\\
\label{v1}\quad=&\,
\sum_{n_{2}=1}^{N}
\sum_{n_{1}=1}^{N}
\sum_{m_{2}=1}^{M}
\sum_{m_{1}=1}^{M}
a_{m_{1} n_{1}} \overline{a_{m_{2} n_{2}}}
\int_{\mathbb{R}}\,\chi(x)\,
e^{ \imath\,(m_{1}-m_{2})\,x} dx\,
\int_{\mathbb{R}}\,\chi(y)\,
e^{ \imath (n_{1}-n_{2}) y} dy.
\end{align}
By the localisation principle
for oscillatory integrals \cite{stein} we have
\begin{align*}
\left|\int \,\chi(x) \,e^{\imath\,\nu x}\,dx \right|
\lesssim \frac{1}{1+\nu^{2}},
 \qquad \nu \in \mathbb{R}.
\end{align*}
Using this in (\ref{v1}) we get
\begin{align}\label{v2}
\parallel V_{M,N}\parallel^{2}_{L^{2,2}}\,\lesssim&\,
\sum_{n_{2}=1}^{N}
\sum_{n_{1}=1}^{N}\,
\frac{1}{1+(n_{1}-n_{2})^{2}}\,
\sum_{m_{2}=1}^{M}
\sum_{m_{1}=1}^{M}
\,
\frac{|a_{m_{1} n_{1}}|\,
|a_{m_{2}\, n_{2}}|}
{1+(m_{1}-m_{2})^{2}}.
\end{align}
Using Young's inequality for the convolution
of sequences we find
\begin{align*}
\sum_{m_{2}=1}^{M}
\sum_{m_{1}=1}^{M}
\,
\frac{|a_{m_{1} n_{1}}|\,
|a_{m_{2}\, n_{2}}|}
{1+(m_{1}-m_{2})^{2}}\,\lesssim\,
\left(
\sum_{m_{1}=1}^{M}\,a_{m_{1} n_{1}}^{2}\right)^
{\frac{1}{2}}\,
\left(
\sum_{m_{2}=1}^{M}\,a_{m_{1} n_{1}}^{2}
\right)^{\frac{1}{2}}.
\end{align*}
Plugging this estimate into (\ref{v2})
then reapplying Young's inequality gives the estimate
\begin{align*}
\parallel V_{M,N}\parallel_{L^{2,2}}
\;\lesssim\;\parallel A \parallel_{l^{2,2}.}
\end{align*}

\bigskip
\bigskip
Mathematics Department, Faculty of Science\\  Assiut University, Assiut,71516, Egypt\\
ahmed.abdelhakim@aun.edu.eg


\begin{thebibliography}{99}
\bibitem{BerghLofstrom}
     \newblock J. Bergh and J.
L\"{o}fstr\"{o}m,
     \newblock \emph{Interpolation Spaces. An Introduction},
     \newblock  Springer-Verlag, Berlin, 1976.
\bibitem{baxan}
     \newblock  A. Baxansky and N. Kiryati,
     \newblock Use of a double Fourier series
for three-dimensional shape representation,
\newblock \emph{Computing}, \textbf{88} (2010), 173-191.
\bibitem{damianorem}
     \newblock Damiano Foschi,
     \newblock \emph{ Some remarks on the $L^{p}-L^{q}$  boundedness of trigonometric sums and oscillatory integrals},
     \newblock Communications on pure and applied analysis, \textbf{4} (2005), 569-588.
\bibitem{analysis}
     \newblock  Elliott H. Lieb and Michael Loss,
     \newblock \emph{Analysis},
     \newblock 2nd edition,  Graduate Studies in Mathematics,
vol. 14, American Mathematical Society,
Providence, RI, 2001.
\bibitem{stein} \newblock E. M. Stein, \newblock \emph{Harmonic analysis:
 real-variable methods, orthogonality,
and oscillatory integrals},
\newblock Princeton Mathematical Series,
 43. Princeton University Press, Princeton, NJ, 1993.
\bibitem{taobook}
\newblock T. Tao,
\newblock \emph{Nonlinear Dispersive Equations:
Local and Global Analysis},
\newblock CBMS Regional
Conference Series in Mathematics, 2006.
\bibitem{marsh}
\newblock E. C. Titchmarsh,
\newblock
\emph{The Theory of the
Riemann Zeta-function},
\newblock 2nd ed., revised by D. R. Heath-Brown,
Clarendon Press, Oxford, 1986.
\bibitem{lova1}
\newblock T. M. Vukolova and
M. I. Dyachenko,
\newblock
\emph{Bounds for norms of sums of double trigonometrical series with multiply monotone coefficients}, \newblock
Izv. VUZ. Matematika,
\textbf{38} (1994), No. 7, 20-28.
\bibitem{lova2} \newblock T. M. Vukolova and
M. I. Dyachenko,
\newblock \emph{Estimates for mixed norms of the sums
of double trigonometric series with multiply
monotonous coefficients},
\newblock Izv. VUZ. Matematika,
\textbf{41} (1997), No. 7, 3-13.
\bibitem {samuel}
\newblock Samuel Y. K. Yee,
\newblock
\emph{Solution of Poisson's Equation on
a Sphere by Truncated Double Fourier Series},
\newblock Mon. Wea. Rev., \textbf{109} (1981), 501–505.
\end{thebibliography}
\end{document}